\documentclass[11pt,twoside]{amsart}
\usepackage{amssymb}

\usepackage[latin1]{inputenc}
\usepackage[T1]{fontenc}
\usepackage{mathtools}
\usepackage{graphicx}
\usepackage{tikz}
\usepackage{mathabx}
\usetikzlibrary{chains}
\usepackage[OT2,T1]{fontenc}

\PassOptionsToPackage{pdfusetitle,pagebackref,colorlinks}{hyperref}
\usepackage{bookmark}
\hypersetup{
  linkcolor={red!70!black},
  citecolor={green!70!black},
  urlcolor={blue!80!black}
}

\usepackage[latin1]{inputenc}
\usepackage{amsmath}
\usepackage{amsthm}

\usepackage[all]{xy}
\usepackage{hyperref}
\newtheorem{thm}{Theorem}[section]

\newtheorem{prop}[thm]{Proposition}

\newtheorem{lem}[thm]{Lemma}

\numberwithin{equation}{section}

\theoremstyle{definition}

\newtheorem{remark}[thm]{Remark}

\DeclareFontFamily{U}{mathc}{}
\DeclareFontShape{U}{mathc}{m}{it}%
{<->s*[1.03] mathc10}{}
\DeclareMathAlphabet{\mathcal}{U}{mathc}{m}{it}
\newcommand{\kend}{\mathcal{E\mkern-3mu nd}}

\newcommand{\ch}{{\rm ch}}

\newcommand{\Br}{{\rm Br}}
\newcommand{\SBr}{{\rm SBr}}
\newcommand{\SBro}{{\rm SBr}^{\rm o}}

\newcommand{\per}{{\rm per}}
\newcommand{\ind}{{\rm ind}}
\newcommand{\CH}{{\rm CH}}
\newcommand{\NS}{{\rm NS}}
\newcommand{\Pic}{{\rm Pic}}

\newcommand{\rk}{{\rm rk}}

\newcommand{\cal}{\mathcal}
\newcommand{\ka}{{\cal A}}

\newcommand{\kc}{{\cal C}}

\newcommand{\kl}{{\cal L}}

\newcommand{\ko}{{\cal O}}
\newcommand{\kp}{{\cal P}}

\newcommand{\GG}{\mathbb{G}}

\newcommand{\ZZ}{\mathbb{Z}}
\newcommand{\QQ}{\mathbb{Q}}

\newcommand{\CC}{\mathbb{C}}

\newcommand{\PP}{\mathbb{P}}

\DeclareSymbolFont{cyrletters}{OT2}{wncyr}{m}{n}
\DeclareMathSymbol{\Sha}{\mathalpha}{cyrletters}{"58}

\renewcommand{\to}{\xymatrix@1@=15pt{\ar[r]&}}
\newcommand{\lto}{\xymatrix@1@=15pt{&\ar[l]}}
\renewcommand{\rightarrow}{\xymatrix@1@=15pt{\ar[r]&}}
\renewcommand{\mapsto}{\xymatrix@1@=15pt{\ar@{|->}[r]&}}
\newcommand{\mapslto}{\xymatrix@1@=15pt{&\ar@{|->}[l]&}}
\renewcommand{\twoheadrightarrow}{\xymatrix@1@=18pt{\ar@{->>}[r]&}}
\newcommand{\twoto}{\twoheadrightarrow}

\renewcommand{\hookrightarrow}{\xymatrix@1@=15pt{\ar@{^(->}[r]&}}
\newcommand{\hook}{\xymatrix@1@=15pt{\ar@{^(->}[r]&}}
\newcommand{\congpf}{\xymatrix@1@=15pt{\ar[r]^-\sim&}}
\renewcommand{\cong}{\simeq}

\newcommand{\TBC}[1]{}

\makeatletter
\def\blfootnote{\xdef\@thefnmark{}\@footnotetext}
\makeatother

\begin{document}

\title[Splitting Brauer classes and the period--index problem]{Splitting unramified Brauer classes by abelian torsors and the period--index problem}

\author[D.\ Huybrechts, D.\ Mattei]{Daniel Huybrechts and Dominique Mattei}

\address{Mathematisches Institut \& Hausdorff Center for Mathematics,
Universit{\"a}t Bonn, Endenicher Allee 60, 53115 Bonn, Germany}
\email{huybrech@math.uni-bonn.de, mattei@math.uni-bonn.de}

\begin{abstract} \noindent We use twisted relative Picard varieties to split
Brauer classes on projective varieties over algebraically closed fields
by torsors for a fixed abelian scheme independent of the Brauer class. The construction is also used to prove that the index of an unramified Brauer class
divides a fixed power of its period.

 \vspace{-2mm}
\end{abstract}

\maketitle
\blfootnote{The authors are supported by the ERC Synergy Grant HyperK (ID 854361).}

\section{Introduction} 
We begin by stating the two main results of the paper. 
For motivation and further comments on the history
we refer to Sections \ref{sec:IntroSplit} \& \ref{sec:pip}.

\subsection{} Our first theorem shows that  Brauer classes on a projective variety can be split after pull-back to torsors for abelian varieties,  its proof
will be explained in Section \ref{sec:TwistedPic} and concluded in Section \ref{sec:proof}.

\begin{thm}\label{thm:main}
Let $X$ be an integral projective variety over an algebraically closed field $k$
and let $K(X)$ be its function field. Then there exists an abelian variety $A$ over $K(X)$
such that for every Brauer class $\alpha\in \Br(X)$ one finds an $A$-torsor $B_\alpha$ 
 that splits $\alpha$.
\end{thm}

In other words, for every $\alpha\in \Br(X)$ the composition $$\Br(X)\to \Br(K(X))\to\Br(B_\alpha)$$ maps $\alpha$ to the trivial Brauer class on $B_\alpha$.
Thus, the theorem shows how to split unramified Brauer classes over the
function field $K(X)$ by passing to torsors for a fixed abelian variety.

We emphasise that our construction produces an abelian variety $A$ that
is independent of the Brauer class $\alpha$, but the torsor $B_\alpha$  does depend on $\alpha$. However,  as all the $B_\alpha$  are torsors for the same abelian variety $A$, their dimensions are independent of $\alpha$. Let us also point out that typically $A$ and $B_\alpha$ are not unique and may even come in families.

\subsection{} The second main result concerns the period--index problem.
Recall that the period of a central simple algebra $A$ over a field $K$ is by definition
the order $\per(A)=\per(\alpha)=|\alpha|$ of its class $\alpha=[A]\in\Br(K)$ in the Brauer group of $K$. Its index is defined as $\ind(A)=\ind(\alpha)=\dim(D)^{1/2}$, where
$D$ is the unique division algebra such that $A\cong M_n(D)$ for some $n$.
It is known classically that $\per(D)\mid \ind(\alpha)$ and that both invariants share the same prime factors, i.e.\ $\ind(\alpha)\mid\per(\alpha)^{e_\alpha}$ for some integer $e_\alpha$ which a priori depends on $\alpha$. \smallskip

Our second theorem shows that for unramified Brauer classes over the function field $K(X)$ of a projective variety, i.e.\ for classes in the image of $\Br(X)\to \Br(K(X))$,
the exponent $e_\alpha$ can be chosen uniformly. The proof is presented in Section \ref{sec:proofInd}.

\begin{thm}\label{thm:mainind}
Let $X$ be an integral projective variety over an algebraically closed field.
Then there exists an integer $e$ such that for all unramified classes $\alpha\in \Br(K(X))$ one has  $$\ind(\alpha)\mid\per(\alpha)^{e}.$$
\end{thm}

Roughly, the exponent appearing in the theorem is 
$e=2g(C)$, where $g(C)$ is the genus of  a very ample complete intersection curve on $X$ {and thus does depend on $X$}, cf.\ (\ref{eqn:indper}). This uniform bound for the index is a consequence of the same geometric construction of the torsors 
used to prove Theorem \ref{thm:main}.
\smallskip

{Section \ref{sec:warmup} treats the case of elliptic K3 surfaces, with or without a section. It serves as a warm-up to the general case, as some of the arguments are much more apparent, but it
also proves more specific results in this situation.}


\subsection{Splitting Brauer classes}\label{sec:IntroSplit}
Splitting Brauer classes by passing to higher-dimensional varieties is nothing unusual.
For example, if $\pi\colon P_\alpha\to X$ is a Brauer--Severi variety representing
the Brauer class $\alpha\in\Br(X)$ then $\pi^\ast(\alpha)$
is the trivial class in $\Br(P_\alpha)$. However, the dimension of
$P_\alpha$ depends on $\alpha$ and cannot be {uniformly} bounded.

Since for any subvariety $Y\subset P_\alpha$ the projection $\pi_Y\colon Y\to X$
also splits $\alpha$, i.e.\ $\pi_Y^\ast(\alpha)$ is the trivial class in $\Br(Y)$, one can use hyperplane sections of the projective variety $P_\alpha$ to produce a projective variety $Y_\alpha \twoheadrightarrow X$ of relative dimension one with irreducible generic fibre that splits a given Brauer class $\alpha$. So, the dimension of the variety to split $\alpha$ can be bounded.
However, with this construction the genus of the generic
fibre will be unbounded for varying $\alpha$.

In fact, it is known that if a smooth projective curve $C$ of genus $g$ over a field $K$ splits a Brauer class $\alpha\in \Br(K)$, then  the order of $\alpha$ divides $2g-2$.\TBC{and in fact its index. Add reference. It is not in CT-S or G. It follows from the standard sequence relating Br and Pic, I think.} {Thus, when its genus is at least two, it is bounded  from below by $(1/2)\cdot|\alpha|$.}\smallskip
 
Whether every Brauer class $\alpha\in \Br(X)$, or more generally $\alpha\in \Br(K)$,
is split by a genus one curve $C_\alpha$ is a rather recent question. In this generality, it was first raised by Clark and Saltman, cf.\ \cite{Clark,RuVi}, but {with earlier results for local fields by Roquette \cite{Ro}.} The question is wide open but it was answered positively 
for low index, see \cite{Auel,AntAuel,dJHo}. 
In Section \ref{sec:ellK3} we settle the case of Brauer classes on elliptic K3 surfaces with a section.\smallskip

Note that in light of our result one could ask the even stronger question whether every class $\alpha\in \Br(X)$ is split by a torsor for some fixed elliptic curve $E$ over $K(X)$. However, doubts already about the weaker version
have been raised by de Jong and Ho \cite{dJHo} and it seems likely that Theorem \ref{thm:main} is the best result in this direction one can hope for in general.  {Note however that Antieau and Auel \cite[Thm.\ C]{AA} show that every cyclic Brauer class over a field containing an algebraically closed field is split by a torsor for an elliptic curve.}\smallskip

 {Theorem \ref{thm:main} can be seen as an improvement of a result by Ho and Lieblich \cite{HoLieb}, it essentially answers  their Question 6.0.1
 for unramified classes, and by Antieau and Auel \cite[Thm.\ 3.18]{AA}.} More concretely, 
 in \cite{HoLieb} it is shown that for a fixed field $K$ there exists for every class $\alpha\in\Br(K)$ of index $\not\equiv2\,(4)$ a smooth projective curve $C_\alpha$ such that $\alpha$ is split by its Albanese
 ${\rm Alb}(C_\alpha)=\Pic^1(C_\alpha)$, a torsor for the abelian variety $\Pic^0(C_\alpha)$. If the index is $\equiv 2(4)$, then $\alpha$ is split by $C'\times\Pic^1(C_\alpha)$, where $C'$ is an additional curve of genus one. In both cases, the curve $C_\alpha$ itself is obtained by taking complete intersections in a Brauer--Severi variety representing $\alpha$ as described above. In particular, its genus and hence the dimension of 
 $\Pic^1(C_\alpha)$ depends on the index of $\alpha$ and cannot be bounded.
 {As was explained to us by Auel, splitting by torsors for abelian varieties can also be achieved by using the Merkurjev--Suslin theorem \cite{MS}, cf.\ \cite[Ch.\ 8]{GS}, which shows  that  $\Br(K(X))$ is generated by cyclic classes. Combined with a result by Matzri
 \cite{Matzri} bounding the symbol length of Brauer classes over function fields, the dimension of the abelian variety can be again bounded in terms of the period of the Brauer class.} {The argument in \cite[\S3]{AA} uses instead a result of Raynaud about realizing finite abelian group schemes as subschemes of abelian schemes.}
 \smallskip
 
In  \cite{HoLieb} one finds further questions about splitting Brauer classes. Torsors for elliptic curves and abelian varieties seem geometrically the most interesting cases, but one could try to use K3 surfaces or Calabi--Yau varieties just as well. We have nothing to say about those.

\subsection{Period--index problem}\label{sec:pip}
The period--index problem, asking for an exponent $e_\alpha$
that is independent of $\alpha$, has a long history. 
For number fields, it has been proved by Albert, Brauer, Hasse, and Noether \cite{AH,BNH} that
$\per(\alpha)=\ind(\alpha)$, i.e.\ $e_\alpha=1$, and the same holds for function fields of transcendence degree one over finite fields.

In the geometric situation,
where $K$ is the function field of a {curve} over an algebraically closed field, the question is vacuous by Tsen's theorem. {For surfaces, de Jong \cite{dJ} proved $\per(\alpha)=\ind(\alpha)$,
so again $e_\alpha=1$, for Brauer classes $\alpha$ with $\per(\alpha)$ prime to the characteristic of the ground field, so in particular for all classes in characteristic zero. Again de Jong \cite{dJS} and Lieblich \cite{LiebComp,LiebAnn} later extended the result to all classes.
Already in \cite{CTBr}, Colliot-Th\'el\`ene proposed the bound $e_\alpha=\dim(X)-1$ 
in arbitrary dimensions, see also \cite[\S 2.4]{CTBourb}.}

{In \cite[Conj.\ 1.2]{dJP}, de Jong and Perry formulated a weaker a conjecture
that asks for a bound
$e=e_\alpha$ that only depends on the variety $X$ and not on the (unramified) Brauer class $\alpha$. Moreover, they proved the existence of such uniform bound under the assumption of that the Lefschetz standard conjecture holds true in degree two. Theorem \ref{thm:mainind} proves their conjecture unconditionally.}

{It is known that techniques developed by de Jong and Starr \cite{dJS,Starr} allow one to deduce the conjecture $e_\alpha=\dim(X)-1$ for all Brauer classes from the case of unramified ones. Unfortunately, as our exponent depends on the geometry of $X$ and not only on its dimension, this argument does not work which limits our approach to the case of unramified Brauer classes.}
 \smallskip

{The proof of Theorem \ref{thm:mainind} relies on the construction of a module $V_\alpha$ of bounded dimension over the Azumaya algebra $A$ representing $\alpha$.  This module $V_\alpha$ is eventually constructed as a space of theta functions on the twisted Picard variety $B_\alpha$, see (\ref{eqn:theta}) in Section \ref{sec:alg}.}

\smallskip

{After the first version of our paper appeared on the arxiv, an alternative argument for Theorem \ref{thm:mainind} was communicated to us 
by M.\ Lieblich and also by B.\ Antieau (coming out of an independent discussion with A.\ Auel).
Colliot-Th\'el\`ene informed us that also D.\ Krashen, has independently found a proof
for the existence of a uniform bound.

The geometric setup is  the same.  
The tensor product $E\mapsto E^{\otimes \per(\alpha)}$ is used to define a morphism $\Pic^d_\alpha(C)\to \Pic^0(C)$ from a twisted Picard variety to an untwisted one. 
In our approach, the module $V_\alpha$ alluded to above is constructed as the space of global sections of the pull-back of the theta divisor twisted by the universal Poincar\'e bundle. The argument simplifies considerably, if instead restriction to the finite kernel of the morphism 
is used. In particular, the argument avoiding the theta bundle leads then directly to the exponent $2g(C)$ without 
the further complication of having to ensure the existence of a universal twisted bundle 
as in Sections \ref{sec:nosect} \& \ref{sec:End}. 
}

\subsection{Twisted Picard varieties} The key idea to approach Theorem \ref{thm:main} is very simple. We
use a dominating universal family of smooth (complete intersection) curves $\kc\to T$ in $X$ and consider the twisted relative Picard variety 
$\Pic_\alpha^d(\kc/T)$. These  twisted Picard varieties are special cases of moduli spaces of twisted sheaves constructed by Lieblich \cite{LiebDuke} and Yoshioka \cite{Yosh}, but  can also be viewed as instances of Simpson's moduli space
of $\Lambda$-modules \cite{Simpson}, see also \cite{HS} for the case of positive characteristic. They have recently been studied for K3 surfaces in \cite{HuyMa} and we use the notation introduced there. For an appropriate choice of $d\in\QQ$, the projection $\Pic_\alpha^d(\kc/T)\to T$ is a torsor for the abelian scheme $\Pic^0(\kc/T)\to T$ and in the ideal
situation $\Pic_\alpha^d(\kc/T)$, as a moduli space of twisted sheaves on $X$, is  fine. Since the universal family over
$\kc\times_T\Pic_\alpha^d(\kc/T)$ is twisted with respect to the pull-back of $\alpha$ and it is of rank one, the pull-back of $\alpha$ is split. \smallskip

The final step is to shrink $\kc\to T$ to $\kc'\to T'$ such that
the projection $\kc'\to X$ is birational. Then the projection $\kc'\times_{T'}\Pic_\alpha^d(\kc'/T')\to\kc'\to X$ provides the torsor claimed in the theorem.\smallskip

Let us be a bit more explicit about the moduli space that we wish to exploit. As a first step, we pick an Azumaya algebra
$\ka$ on $X$ that represents the Brauer class $\alpha$. Then  $\Pic^d_\alpha(\kc/T)$
is  constructed as an open subset of a moduli space of modules over $\ka$ supported on curves $\kc_t\subset X$, $t\in T$, cf.\ \cite{Simpson}. The universal family is then a sheaf $\kp$ on $\kc\times_T\Pic_\alpha^d(\kc/T)$ which is a module over the pull-back $\ka'$ of $\ka$
under the projection to $X$. By definition of the twisted Picard variety as a moduli space
of certain sheaves, the rank
of $\kp$ is $d(\ka)\coloneqq\sqrt{\rk(\ka)}$. This implies that the pull-back of $\alpha$ is trivial,
for $\ka'\cong\kend(\kp)$ on a dense Zariski open subset.
\smallskip

There are various ways of dealing with moduli spaces in a twisted situation. Each of them
involves an additional choice, e.g.\ of a \v{C}ech cocycle, an Azumaya algebra,
 a Brauer--Severi varieties, or $\mu_n$-gerbe, representing the fixed Brauer class.
The moduli space is then constructed as a space of sheaves twisted with respect to the
\v{C}ech cocycle, of modules over the Azumaya algebra, of sheaves
on the Brauer--Severi variety, or of {invertible sheaves of weight one on a $\mu_n$-gerbe.} We decided to represent the Brauer class by
an Azumaya algebra but every other choice would have worked just as well.\smallskip

In the end the triviality of the pull-back $\alpha'$ of $\alpha$ follows from the existence of a locally free $\alpha'$-twisted sheaf of rank one or, in terms of Azumaya algebras, from the existence of a locally free $\ka'$-module of rank $d(\ka)$,
where $\ka$ represents $\alpha$ and $\ka'$ is its pull-back.

\medskip

\noindent
{\bf Acknowledgements:} We wish to thank Asher Auel
for inspiration and for help with the literature, as well as for comments on the first version of the paper. We also grateful to Evgeny Shinder for
critical comments on the first draft and to Gebhard Martin for interesting questions.
The first author gratefully acknowledges the hospitality of the ITS-ETH Zurich during his stay in the spring of 2023.

\section{Warmup: Elliptic curves on K3 surfaces}\label{sec:warmup}

Before presenting the general proof of Theorem \ref{thm:main}, we discuss the special case of elliptic K3 surfaces. It is the first instance where universal sheaves are seen to split Brauer classes and it is instructive to study this case first before dealing with the general situation in the next section. As it seems more customary, we here work with twisted sheaves instead of modules over Azumaya algebras, but this does not effect the essence of the argument.

We will find that for elliptic surfaces with a section all Brauer classes are split by a genus one fibration (over a fixed elliptic curve) and so Clark's original questions has an affirmative answer in this case, see Proposition \ref{prop:warmup}. For elliptic K3 surfaces without a section, only Brauer classes of an order coprime to the multi-section index of the elliptic fibration  can be split by the same method, see Proposition \ref{prop:K3nosect}.

We will conclude this section by discussing families of (singular) K3 surfaces. Something can still be said but the result is less compelling, see Proposition \ref{prop:singell}.\smallskip

{In fact, the discussion in this section applies to elliptic surfaces, with and without a section,
that are not necessarily K3 surfaces. However, for simplicity and
as the main purpose of this section is to illustrate the proof in the general setting, we stick to K3 surfaces.}

\subsection{Elliptic K3 surfaces with a section}\label{sec:ellK3} Consider an elliptic K3 surface $S_0\to\PP^1$ with a section and a class
$\alpha\in\Sha(S_0/\PP^1)$. By definition of the Tate--{\v{S}}afarevi{\v{c}} group
$\Sha(S_0/\PP^1)$,
the class $\alpha$ corresponds to an elliptic K3 surface $S\to\PP^1$ (without a section)
together with an isomorphism $\Pic^0(S/\PP^1)\cong S_0$ relative over $\PP^1$. Here,
$\Pic^0(S/\PP^1)$ denotes the minimal smooth compactification of the Jacobian
fibration of the family of smooth fibres of $S\to \PP^1$. Such a compactification is
provided by the moduli space $M(0,f,0)$ of stable sheaves on $S$ with Mukai vector $(0,f,0)$, where
$f$ is the class of a fibre, cf.\ \cite[Ch.\ 11]{HuyK3}. However, such a moduli space
is  not fine, i.e.\ there is no universal family on $S_0\times S$. 

The obstruction for a universal family to exist is a class in $\Br(S_0)$. In fact, this class
is nothing but $\alpha$ itself under the well-known isomorphism $\Sha(S_0/\PP^1)\cong\Br(S_0)$.
To be more precise, the obstruction class in $\Br(S_0)$ is the class of
the \v{C}ech cocycle that comes up naturally when we try to glue the local (\'etale or analytic) universal families, which always exist. This point of view also shows that a universal family $\kp$ exists as a twisted sheaf
$\kp$ on $S\times S_0$, where the twist is with respect to the pull-back of the cocycle on
$S_0=\Pic^0(S/\PP^1)$ representing the obstruction class $\alpha$. The construction was systematically studied by C\v{a}ld\v{a}raru \cite{Cald}.\smallskip

By definition of the moduli space, the $(1\times \alpha)$-twisted sheaf $\kp$
on $S\times S_0$ has support on the closed subscheme $S\times_{\PP^1}S_0\subset S\times S_0$. Furthermore, there it is of rank one. Hence, on a non-empty Zariski open subset
 $V\subset S\times_{\PP^1}S_0$, the twisted sheaf $\kp$ is locally free of rank one
which implies that $(1\times\alpha)|_V$ is a trivial Brauer class.

The generic fibre of the second projection $S\times_{\PP^1}S_0\to S_0$ is a smooth
curve of genus one. More precisely, it is  the generic fibre
$S_\zeta$ of $S\to \PP^1$  base changed to the generic point $\eta\in S_0$ which, of course, maps to $\zeta$ under the projection $S_0\to \PP^1$. We have proved the following result.

\begin{prop}\label{prop:warmup}
Let $S_0\to \PP^1$ be an elliptic K3 surface with a section and let $\alpha\in \Br(S_0)$. Then there exists a genus one curve $C_\alpha$ over its function field $K(S_0)$ such that the image  of $\alpha$ under the composition
$\Br(S_0)\to \Br(K(S_0))\to\Br(C_\alpha)$ is trivial.\qed
\end{prop}

We can be more precise about the curve $C_\alpha$. By construction, it
is the generic fibre of $S\times_{\PP^1}S_0\to S_0$ which is the base
change to $K(S_0)$ of the moduli space of line bundles on the generic fibre of $S_0\to \PP^1$ twisted with respect to the restriction of $\alpha$. But this moduli space is naturally a torsor for $\Pic^0$ of the generic fibre of $S_0\to \PP^1$ where the action is given by tensor product.

Since $S_0\to \PP^1$ has a section, its relative $\Pic^0$ is just $S_0\to\PP^1$ itself. Hence, all the genus one curves $C_\alpha$
in Proposition \ref{prop:warmup} are torsors for the same elliptic curve, namely the 
base change to $K(S_0)$ of the generic fibre of $S_0\to \PP^1$.

\subsection{Elliptic K3 surfaces without a section}\label{sec:nosect}
Since in the above discussion, the K3 surface $S$ could be seen as
a certain moduli space $\Pic^d_\alpha(S_0/\PP^1)$ of twisted sheaves on the fibres
of $S_0\to \PP^1$, cf.\ \cite{HuyMa}, one could try to extend the argument to
elliptic K3 surfaces $S_0\to \PP^1$ without a section by simply considering $\Pic^d_\alpha(S_0/\PP^1)$. In \cite{HuyMa} we explain why
the notation $\Pic^d_\alpha(S_0/\PP^1)$ makes only sense when a
lift of $\alpha$ to an element in the special Brauer group $\SBr(S_0)$ is chosen: Only then the degree $d$ is well defined. This subtlety is of no importance in the discussion here and we will just ignore it, but see Section 
\ref{sec:introgen}.\smallskip

So the idea would be to use a universal family
$\kp$ on $\Pic^d_\alpha(S_0/\PP^1)\times_{\PP^1}S_0$ to argue that
the pull-back of $\alpha$ under the second projection has to be trivial. And indeed, by definition of $\Pic_\alpha^d(S_0/\PP^1)$ as a moduli space of $\alpha$-twisted sheaves, $\kp$ would be twisted with respect to $\alpha$ on the second factor.  

However, there is no guarantee that such a universal family $\kp$ exists as
a $(1\times\alpha)$-twisted sheaf, i.e.\ that the moduli space $\Pic^d_\alpha(S_0/\PP^1)$ is a fine moduli space of twisted sheaves on $S_0$. In other words, $\kp$ might only exist as a sheaf
that is not only twisted by $\alpha$ on the second factor but also by some
obstruction class on the first. If this is the case, the argument breaks down and we cannot conclude  the triviality of the pull-back of $\alpha$.

But something can be salvaged from this approach by exploiting a certain flexibility in the choice of $d$ (after choosing a lift of $\alpha$ to a class in the special Brauer group).
For this we make use of  a well-known criterion for 
$\Pic^d_\alpha(S_0/\PP^1)=M_\alpha(0,f,d)$ to be a fine moduli space:
It suffices to find an $\alpha$-twisted locally free sheaf $E$ such that $\chi(F\otimes E^\ast)=1$ for $F\in \Pic^d_\alpha(S_0/\PP^1)$. See
\cite[Ch.\ 4.6]{HuLe} for the untwisted case, the proof in the twisted case is identical. 

\begin{lem} Assume the order of $\alpha\in \Br(S_0)$ is coprime to the generator $m$ of $(\NS(S_0).f)$. Then there exists a (twisted) degree $d\in \QQ$ with
$\Pic^d_\alpha(S_0/\PP^1)$ non-empty and an $\alpha$-twisted locally free sheaf $E$ on $S$ such that $\chi(F\otimes E^\ast)=1$
for every $F$ parametrised by  $\Pic^d_\alpha(S_0/\PP^1)$. In particular, the moduli space $\Pic_\alpha^d(S_0/\PP^1)$ of $\alpha$-twisted sheaves on $S_0$ is fine.
\end{lem}

\begin{proof}
We start with any $\alpha$-twisted line bundle $F$ supported on a smooth fibre $f$ of
$S_0\to\PP^1$ that defines a point in some non-empty moduli space $\Pic_\alpha^d(S_0/\PP^1)$. Let us also fix  some $\alpha$-twisted
locally free sheaf $E$ on $S_0$ of rank $|\alpha|$. The existence of the latter is guaranteed by de Jong's solution \cite{dJ}  of the period-index problem for algebraic surfaces. 

We will now modify $E$, keeping it $\alpha$-twisted
and locally free of rank $|\alpha|$,  and $F$, possibly changing $d$ in the process, such that
eventually $\chi(F\otimes E^\ast)=1$. 

Observe first that the kernel $F'$ of any surjection $F\twoheadrightarrow k(x)$, $x\in f$, defines again an $\alpha$-twisted  line bundle on $f$ which satisfies
$\chi(F'\otimes E^\ast)=\chi(F\otimes E^\ast)-|\alpha|$. Second, we pick a curve $C\subset S_0$ with $(C.f)=m$\footnote{It is enough that $(C.f)=k\cdot m$ with also $k$ coprime to $|\alpha|$, which is slightly easier to arrange.} and consider an elementary transformation $E'$ of $E$ along $C$, i.e.\ $E'$ is the $\alpha$-twisted
locally free sheaf given as the kernel of some surjection 
$E\twoheadrightarrow \kl$ onto some $\alpha$-twisted line bundle $\kl$ on $C$. Then $\chi(F\otimes E'^\ast)=\chi(F\otimes E^\ast)-m$.

Since $|\alpha|$ and $m$ are coprime, applying the procedure multiple times
eventually provides us with $E$ and $F$ as claimed.
\end{proof}

Note that in particular we reprove the assertion in the case that $S_0\to \PP^1$ has section,
for then  $(\NS(S_0).f)=\ZZ$ and hence $m=1$. If there is no section, then
the discussion eventually leads to the following generalisation of Proposition
\ref{prop:warmup}.

\begin{prop}\label{prop:K3nosect}
Let $S_0\to \PP^1$ be an elliptic K3 surface and let $m\in \ZZ$ be such that
$m\cdot\ZZ=(\NS(S_0).f)$ for the fibre class $f$. Then one finds an elliptic curve $E$ over $K(S_0)$ such that for every $\alpha\in\Br(S_0)$
of order $|\alpha|$ coprime to $m$ there exists a torsor $C_\alpha$ for $E$ over $K(S_0)$ such that the image of $\alpha$ under $\Br(S_0)\to\Br(K(S_0))\to \Br(C_\alpha)$ is trivial.\qed
\end{prop}

\subsection{Families of elliptic curves on arbitrary K3 surfaces}\label{sec:genK3}
Every K3 surface admits a covering family of elliptic curves. Does this mean that Brauer classes on arbitrary K3 surfaces are split by curves of genus one? Unfortunately, our approach fails in this generality, but something can still be said.\smallskip

For any K3 surface $S$ there exists a one-dimensional smooth
family of curves of genus one $\kc\to T$ with a dominant map
$\kc\twoheadrightarrow S$. Following the same strategy as in the last two sections, one considers $\Pic_\alpha^d(\kc/T)\times_T\kc\to\kc\to S$.
As soon as $\Pic_\alpha^d(\kc/T)$ is a fine moduli space, which can again
be phrased as a numerical condition on $|\alpha|$ being coprime
to a fixed integer $m$, there exists a universal family of $\alpha$-twisted
sheaves on $\Pic_\alpha^d(\kc/T)\times_T\kc$. This, as before, implies
that the pull-back of $\alpha$ is trivial. Now, $$\Pic_\alpha^d(\kc/T)\times_T\kc\to \kc$$ is a family of curves of genus one, generically a torsor
for $\Pic^0(\kc/T)\times_T\kc\to\kc$.

The difference to the case of elliptic K3 surfaces, with or without a section,
is that the projection $\kc\to S$ is typically not birational. In other words, there often
exists more than one elliptic curve passing through a generic point in $S$. Thus, by this method, one only proves the following.

\begin{prop}\label{prop:singell}
Let $S$ be a complex projective K3 surface. Then there exists
an integer $m>0$, a finite extension $K'/K(S_0)$, and an elliptic curve $E$ over
$K'$ such that for every Brauer class $\alpha\in \Br(S)$ of order $|\alpha|$ coprime to $m$, one finds a torsor $C_\alpha$ for $E$ over $K'$
such that under $\Br(S)\to\Br(K(S))\to \Br(K')\to \Br(C_\alpha)$ the image
of $\alpha$ is trivial.\qed
\end{prop}

Note that any Brauer class splits under some finite extension of $K(S_0)$. So the interest of the last proposition stems from the fact that $K/K(S_0)$ is a fixed
finite extension that works for all Brauer classes (of order coprime to $m$).\smallskip

{
\begin{remark}\label{rem:gen1}
Apart from elliptic surfaces, we do not know of any other types of surfaces
for which covering families of (singular) elliptic curves have been studied,
but the techniques here would certainly also apply to those and prove splitting
of Brauer classes by genus one curves over a fixed finite extension of the function field. See also Remark \ref{rem:gen2}.
\end{remark}}

\subsection{Period--index problem for elliptic K3 surfaces with a section}
{We now outline the proof of Theorem \ref{thm:mainind} in the case of 
elliptic K3 surfaces with a section, which might again serve as an illustration of the argument in the general situation. Here, we use the language of twisted sheaves while in the actual proof in Section \ref{sec:proofInd} we employ sheaves over Azumaya algebras.\smallskip

As in Section \ref{sec:ellK3}, we assume that $\pi\colon S_0\to \PP^1$ is an elliptic K3 surface with a section and that $S\to\PP^1$ is the elliptic K3 surface corresponding to a fixed class $\alpha\in\Br(S_0)\cong\Sha(S_0/\PP^1)$.  
By definition of the Tate--{\v{S}}afarevi{\v{c}} group, $\Pic^0(S/\PP^1)\cong S_0$,
but the Poincar\'e bundle $\kp$  on the fibre product $S\times_{\PP^1}S_0$ only exists as  $(1\times\alpha)$-twisted invertible sheaf $\kp$. Via the second projection $p\colon S\times_{\PP^1}S_0\to S_0$, this fibre product is viewed (over a dense open subset of $S_0$) as a family of smooth curves of genus one (the smooth fibres of $S\to \PP^1$) and $\kp$ is a family $\kp|_{S_{\pi(x)}}$, $x\in S_0$, of untwisted line bundles of degree zero on these fibres. 

It is is known that the period $\per(\alpha)=|\alpha|$ coincides with the minimal fibre degree of invertible sheaves on $S\to\PP^1$, cf.\ \cite[Ch.\ 11.4]{HuyK3}.
Thus, there exists a line bundle $M$ on $S$ such that $\deg(M|_{S_t})=\per(\alpha)$. Hence,  the restrictions of $\kp\otimes(M\boxtimes \ko_{S_0})$  to the fibres of $p$ are invertible sheaves of degree $m$ and, therefore, the direct image $p_\ast\kp$ is a torsion free, $\alpha$-twisted sheaf of rank $\per(\alpha)$. This immediately
implies that $\ind(\alpha|_{K(X)})=\per(\alpha)$.\smallskip

This proves the following special case of de Jong's result \cite{dJ} for   elliptic K3 surfaces with a section.

\begin{prop} Assume $S_0\to\PP^1$ is an elliptic K3 surface with a section. Then
$$\per(\alpha)=\ind(\alpha)$$
for every class $\alpha\in \Br(S_0)\subset\Br(K(S_0))$.\qed
\end{prop}

}
\section{Twisted relative Picard varieties}\label{sec:TwistedPic}
Before entering the proof of Theorem \ref{thm:main}, let us rephrase 
the result geometrically. The assertion is that there
exists a variety ${\bf A}\twoheadrightarrow X$ which generically is an abelian scheme over $X$ such that for every Brauer class $\alpha\in \Br(X)$ one finds a variety ${\bf B}_\alpha\twoheadrightarrow X$  generically
a torsor for ${\bf A}$ such that $$\Br(X)\to\Br({\bf B}_\alpha)$$ maps $\alpha$ to a trivial class on ${\bf B}_\alpha$.

The passage from Theorem \ref{thm:main} to the more geometric version is obtained by viewing $B_\alpha$ as the generic fibre of  ${\bf B}_\alpha$ dominating $X$.
In characteristic zero, we can assume that ${\bf B}_\alpha$ is smooth (locally factorial would be enough) and projective. Indeed, then the restriction map $\Br({\bf B}_\alpha)\to \Br(B_\alpha)$ to the generic fibre is injective. Otherwise, if ${\bf B}_\alpha$ cannot be resolved by a locally factorial variety, one  might have to restrict to a Zariski open neighbourhood of the generic fibre $B_\alpha$, so that ${\bf B}_\alpha$ would not necessarily be projective anymore.\smallskip

{After introducing the setup, Section \ref{sec:introgen} proves a version
of Theorem \ref{thm:main}, namely the existence of $A$ and $B_\alpha$
not over the function field of $X$, but over some universal extension of it obtained
by the family of complete intersection curves. In fact, the argument is explained
first under the additional assumption that the moduli spaces used in the proof are fine. In Section \ref{sec:finems}, we then show how to turn this assumption into a numerical condition on the period of the Brauer class and in Section \ref{sec:proof} how to avoid it altogether. Section \ref{sec:reducing} explains
how the universal family of all complete intersection curve is cut down to a family that dominates $X$ birationally.}
\subsection{Setting: Complete intersection curves}\label{sec:setup}
We now set the stage for the proof of Theorem \ref{thm:main}. So
we let $X$ be an integral projective variety of dimension $n$ over an algebraically closed field $k$.
Passing to its normalisation, we may assume from the start that $X$ itself is normal. In characteristic zero we can even assume that $X$ is smooth,
but we will not need this.  \smallskip

{By replacing $X$ by its blow-up in a smooth closed point $x\in X$ and, letting $Y$ be the exceptional divisor, we can assume that there exists one (smooth) hypersurface $Y\subset X_{\rm sm}$ to which all Azumaya algebras $\ka$ on $X$ restrict trivially or, even stronger, such that the Brauer group of $Y$ is trivial.}\smallskip

Next, we fix a very ample linear system $|h| \cong\PP^N_k$ on $X$
and produce a family of complete intersection curves
  $$\xymatrix@M=7pt{\ar@{->>}[d]~\kc\ar@{^(->}[r]&U\times X\ar[r]&X\\
U&&}$$
of smooth curves $\kc_t\subset X$ parametrised by an open dense subset $U$
of the Grassmannian ${\mathbb G}(n-2,|h|)={\rm Gr}(n-1,H^0(X,\ko(h)))$.  Later it will be convenient to assume that all
curves parametrised by $U$ are contained in the smooth part of $X$, for which we will need the normality of $X$. Also note that we may assume $\dim(X)\geq2$,
so that we really can talk about a family of curves contained in $X$. Indeed,
the Brauer group of a smooth projective curve over an algebraically closed field is trivial.

The relative Picard scheme $\Pic^0(\kc/U)\twoto U$ of the family will be central for
our discussion and eventually leads to the abelian variety $A$.

\subsection{Twisted relative Picard variety}\label{sec:introgen}
For any class $\alpha\in\SBr(X)$ in the special Brauer group $\SBr(X)$\footnote{The difference between the Brauer group $\Br(X)$ and the special Brauer
group $\SBr(X)$ is not essential for our discussion. Lifting a class in $\Br(X)$
to a class in $\SBr(X)$ is only needed to be able to properly talk about the twisted
Picard variety $\Pic^d_\alpha$. {In the language of gerbes, it amounts to lifting a $\GG_m$-gerbe of order $n$ to a $\mu_n$-gerbe.}}
and any $d\in \QQ$, we can consider the relative twisted  Picard variety, cf.\ \cite{HuyMa},
$$\xymatrix{\Pic^d_\alpha(\kc/U)\ar[r]& U,}$$
which, if not empty, is a torsor for the Picard scheme $\Pic^0(\kc/U)\to U$. Its pull-back under $\kc\twoto U$ provides us with a torsor $$\Pic^d_\alpha(\kc/U)\times_U\kc\to \kc
$$ for the
abelian scheme $\Pic^0(\kc/U)\times_U\kc\to\kc$. Altogether, we have the commutative diagram
\begin{eqnarray}\label{eqn:1stdiag}
\xymatrix{\Pic^d_\alpha(\kc/U)\times_U\kc\ar[d]\ar[r]& \kc\ar[d]\ar[r]&X\\
\Pic^d_\alpha(\kc/U)\ar[r]&U.&}
\end{eqnarray}

\smallskip

Let us spell this out a bit more. First pick an Azumaya algebra $\ka$ on $X$ representing the class $\alpha\in\SBr(X)$. Then $\Pic^d_\alpha(\kc/U)$ is an open subset of the
moduli space of stable $\ka$-modules $E$ on $X$ supported on complete
intersection curves $C=H_1\cap\ldots \cap H_{n-1}\subset X$ with $H_i\in |h|$.\smallskip

{

Moduli spaces of this type are a special case of  Simpson's
moduli spaces of $\Lambda$-modules, cf. \cite{Simpson}.
Alternative construction via  gerbes or Brauer--Severi varieties are due to
Lieblich \cite{LiebDuke} and Yoshioka \cite{Yosh}. There are some subtleties
in comparing the stability conditions, but for locally free
 $\alpha$-twisted sheaves  of rank one or, equivalently, $\ka$-modules
of rank $d(\ka)=\sqrt{\rk(\ka)}$ those can be safely ignored.\smallskip

These moduli spaces, e.g.\ of $\Lambda$-modules, are projective once
the Hilbert polynomial is chosen.  The Hilbert polynomial $P$ to be fixed that
would result in singling out $\Pic^d_\alpha(\kc/U)$ is $P(\ell)=d(\ka)\cdot (\ell\cdot (C.h)+d)+(1-g(C))$. Then for a closed point $[F]\in\Pic^d_\alpha(\kc/U)$ corresponding to an $\ka|_C$-module $F$ on a smooth curve $[C]\in U$ we have 
$\chi(F\otimes\ko(\ell h))=P(\ell)$ or, equivalently,  $\ch(F)=d(\ka) \cdot(1+
d\cdot[\text{pt}])$ for the Chern character of $F$ on $C$. 

\begin{remark} On a smooth projective variety, one could alternatively work
with twisted Chern characters on $X$ which one would then fix as $$\ch_\ka(E)=[C]+d\in (\CH^{n-1}(X)\oplus \CH^{n}(X))/_{\sim_{\text{num}}}\otimes\QQ.$$ For $k=\CC$ one can  just work with singular cohomology and fix $\ch_\ka(E)$ as a class in the cohomology $H^{2n-2}(X,\QQ)\oplus H^{2n}(X,\QQ)$. Recall from \cite{HuyMa} that the twisted Chern character $\ch_\ka(E)$ is defined as $\sqrt{\ch(\ka)}^{-1}\cdot\ch(E)$, which for sheaves concentrated on curves is just $d(\ka)^{-1}\cdot\ch(E)$.
\end{remark}}

Note that every $F$ parametrised by $\Pic^d_\alpha(\kc/U)$ is
a locally free $\ka$-module on some smooth complete intersection
curve $C$ with $\rk(F)=d(\ka)$. In particular, $\ka|_C\cong\kend(F)$.
In the language of twisted sheaves, $F$ would correspond to a line bundle
that is twisted with respect to the restriction of a cocycle representing $\alpha$.\smallskip

Assume that $d$ can be chosen such that $\Pic^d_\alpha(\kc/U)$ is a fine moduli space, i.e.\ that there exists a universal sheaf $\kp$ on $\kc\times_U\Pic^d_\alpha(\kc/U)$. 
By definition, $\kp$ is a module over the pull-back $\ka'$
of $\ka$ via the projection 
\begin{equation}\label{eqn:proj}
\xymatrix{\Pic^d_\alpha(\kc/U)\times_U\kc\ar@{->>}[r]&\kc\ar[r]&X.}
\end{equation}
The class $\alpha'\in \SBr(\kc\times_U\Pic^d_\alpha(\kc/U))$ of  $\ka'$
is then the pull-back of $\alpha$, i.e.\ its image
 under 
$$\xymatrix@R=8pt{\SBr(X)\ar[r]& \SBr(\kc)\ar[r]&\SBr(\Pic^d_\alpha(\kc/U)\times_U\kc).\\
\alpha\ar@{|->}[r]&\alpha_\kc\ar@{|->}[r]&\alpha'\phantom{GGGGGGGGG}}$$
We will use the same notation for the corresponding classes in the Brauer groups
$\Br(X)$, $\Br(\kc)$, and $\Br(\Pic^d_\alpha(\kc/U)\times_U\kc)$.

  Note that the restriction of $\kp$ to $[F]\times C\subset \Pic^d_\alpha(\kc/U)\times_U\kc$ is the sheaf $F$ on $C$. Thus, $\kp$ is a locally
free $\ka'$-module of rank $d(\ka')$ and, therefore, the natural map
 $\ka'\congpf\kend(\kp)$ is an isomorphism. Hence, the image
 $\alpha'=[\ka']\in \Br(\Pic^d_\alpha(\kc/U)\times_U\kc)$  of $\alpha\in\Br(X)$ is trivial.\smallskip
 
 In other words, the pull-back $\alpha_\kc\in \Br(\kc)$ is split by 
 $\Pic_\alpha^d(\kc/|U|)\times_U\kc\to \kc$ which is a torsor for the abelian scheme $\Pic^0(\kc/U)\times_U\kc\to\kc$.
Equivalently, if $\eta_\kc\in\kc$ and $\eta_X\in X$ denote the generic points of $\kc$
and $X$, and so in particular $\eta_\kc$ maps to $\eta_X$ under the projection $\kc\to X$, the composition
$$\Br(X)\to \Br(K(X))=\Br(k(\eta_X))\to \Br(k(\eta_\kc))\to \Br(B_\alpha^d)$$ sends $\alpha$ to the trivial class. Here, $B_\alpha^d$ is the generic fibre of $\Pic^d_\alpha(\kc/U)\times_U\kc\twoheadrightarrow\kc$ which can also be seen as the base change of the generic
fibre of $\Pic_\alpha^d(\kc/U)\to U$ to $\eta_\kc$. Thus, (\ref{eqn:1stdiag}) is complemented by
$$\xymatrix@R=13pt{B_\alpha^d\ar@{}[d]|-{\bigcap}\ar[r]&\eta_\kc\ar@{}[d]|-{\rotatebox[origin=c]{270}{$\in$}}\ar@{|->}[r]&\eta_X\ar@{}[d]|-{\rotatebox[origin=c]{270}{$\in$}}\\
\Pic_\alpha^d(\kc/U)\times_U\kc\ar[r]&\kc\ar[r]&X.}$$
Note that the field
extension $K(X)\subset k(\eta_C)$ is independent of $\alpha$.
\subsection{Reducing the family of curves}\label{sec:reducing}
Even under the assumption that there exists a universal $\ka'$-module $\kp$, our discussion does not yet prove Theorem \ref{thm:main}, because $B_\alpha^d$ is a variety over the function field $k(\eta_\kc)$ of $\kc$ and not over $K(X)$. This will be addressed now.\smallskip

{
The very ample linear system $|h|$ embeds $X$ into a projective space
$\PP^N$. We pick a generic linear subspace $P\coloneqq \PP^{N-n}\subset \PP^N$, where
as before $n=\dim(X)$, and project from $P\subset \PP^{N}$ onto a generic
$\PP^{n-1}$. This defines a rational map $\xymatrix{\PP^N\ar@{-->}[r]&\PP^{n-1}}$, which is resolved by the blow-up in $P\subset\PP^{N}$ to a morphism
${\rm Bl}_{P}(\PP^N)\to\PP^{n-1}$. Restricting to $X$, we obtain a morphism
\begin{equation}\label{eqn:blow}
{\rm Bl}_{\{x_i\}}(X)\to\PP^{n-1}
\end{equation} from the blow-up of $X$ in the points of intersections $\{x_1,\ldots,x_d\}=X\cap P$, where $d=\deg(X)$. The fibre of (\ref{eqn:blow})
over a point $t\in\PP^{n-1}$ is the complete intersection curve
$X\cap \overline{Pt}$. Restricting to those points with smooth fibres
leads to
\begin{equation}\label{eqn:fambirat}
\xymatrix@R=20pt{\kc_0\ar[d]\ar@{^(->}[r]&{\rm Bl}_{\{x_i\}}(X)\ar[d]\ar@{->>}[r]&X\\
U_0\ar@{^(->}[r]& \PP^{n-1}.&}
\end{equation}

In more invariant terms, the situation can
be described by $P=\PP(W)\subset\PP^N=\PP(V)$ and $\PP^{n-1}=\PP(W_0)$, with $V=H^0(X,\ko(h))^\ast$
of dimension $N+1$,  a generic linear subspace $W\subset V$ of dimension $N-n+1$,  and a subspace $W_0\subset V$ of dimension $n$ for which the projection $W_0\,\hookrightarrow V\twoto V/W$ is an isomorphism. Mapping a point
$t=[v]\in \PP(W_0)$ to the subspace of all sections $s\in V^\ast=H^0(X,\ko(h))$
vanishing on the subspace spanned by $W$ and $v\in V$ defines an embedding
$$\PP^{n-1}=\PP(W_0)\,\hookrightarrow {\mathbb G}(n-1,|h|).$$ 
We can think of the family $\kc_0\to U_0\subset\PP(W_0)$ as the pull-back of $\kc\to U\subset{\mathbb G}(n-1,|h|)$.

What has been achieved by this construction is that the family 
(\ref{eqn:fambirat}) of curves in $X$ dominates $X$ birationally, i.e.\ there exists exactly one curve $\kc_t$ passing through the generic point of $X$. In particular,
the function fields of $\kc_0$ and $X$ can be identified: $K(\kc_0)=k(\eta_{\kc_0})=k(\eta_X)=K(X)$. Another feature of the construction is that the projection $\kc_0\to U_0$ admits a section. Indeed,  each of the exceptional divisors $E_i\subset {\rm Bl}_{\{x_i\}}(X)$ over the point $x_i\in X$
defines one.\smallskip
}

{\begin{remark}
The idea to use families of complete intersection curves to study Brauer classes is not new. For example, Colliot-Th\'el\`ene  \cite[Lem.\ 1]{CT} proves 
the existence of (\ref{eqn:fambirat})  and Lieblich \cite[Lem.\ 4.2.1.1]{LiebComp}
shows how to relax the assumption on the ground field. \end{remark}}

The restriction of (\ref{eqn:proj}) to $\kc_0$, which is nothing but
\begin{equation}\label{eqn:proj2}\Pic^d_\alpha(\kc_0/U_0)
\times_{U_0}\kc_0\to\kc_0\to X,
\end{equation} can now
be viewed as a torsor for the abelian scheme 
\begin{equation}\label{eqn:proj3}\Pic^0(\kc_0/U_0)\times_{U_0}\kc_0\to \kc_0.
\end{equation}
Since the projection $\kc_0\to X$ is a birational morphism, the generic fibres  of (\ref{eqn:proj2}) and of (\ref{eqn:proj3}), which we denote
by   $B_\alpha$ resp.\  $A$, can both be viewed
as varieties over $k(\eta_X)=K(X)$.

The rest of the argument is as before. The restriction $\kp_0$ of the universal sheaf
$\kp$ on $\Pic^d_\alpha(\kc/U)\times_U\kc$ to the subscheme
$\Pic^d_\alpha(\kc_0/U_0)\times_{U_0}\kc_0\subset \Pic^d_\alpha(\kc/U)\times_U\kc$
is locally free of rank $d(\ka)$ and a module over the pull-back of the Azumaya algebra representing
the Brauer class $\alpha$ under (\ref{eqn:proj2}).

\subsection{Reducing to fine moduli spaces}\label{sec:finems}
The next step is to ensure that for a given class $\alpha$ 
we can choose a rational number $d\in \QQ$ such that $\Pic^d_\alpha(\kc/U)$ is non-empty and fine. 
It is easy to find a $d$ for which the twisted Picard variety is not empty, but the
existence of a universal family will be possible only under the additional assumption that $|\alpha|$ and the top self-intersection number $m\coloneqq (h^{n-1}.Y)$ are coprime (an assumption that we will get rid of eventually in the next section). The idea is the same as in Section \ref{sec:nosect} for elliptic K3 surfaces without a section.\smallskip

For the non-emptiness consider any smooth curve $C$ parametrised by
$U$ and the restriction $\ka|_C$ of the Azumaya algebra $\ka$. Since $\Br(C)$ is trivial, there exists a locally free sheaf $F$ on $C$ with an isomorphism of algebras $\ka\cong\kend(F)$. In particular, $F$ can be viewed as an $\ka$-module with $\ch(F)=r\cdot [C]+\deg(F)$ with $r^2={\rk(\ka)}$ and, hence, $\ch_\ka(F)=[C]+d$, where $d=\deg(F)/r$.
In other words, $F$ defines a point in $\Pic^d_\alpha(\kc/U)$ over the point
$[C]\in U$ and so $\Pic_\alpha^d(\kc/U)$ is not empty.\smallskip

Let us now turn to the existence of the universal family. As in Section \ref{sec:nosect}, it is more convenient to use the language of twisted sheaves.
We first remark that  there is a more flexible 
version of the criterion for the existence of a universal family than the one already used in Section \ref{sec:nosect}: A universal sheaf $\kp$ on $\Pic_\alpha^d(\kc/U)\times_U\kc$, twisted with respect to the pull-back of $\alpha$, exists if one finds {a formal linear combination $u=\sum a_i[G^i]$ of $\alpha$-twisted
sheaves  on $X_{\rm sm}$ (or a class in the corresponding Grothendieck group) 
with $\chi(F\otimes u)=\sum a_i\cdot\chi(F\otimes G^i)=1$ for some $F\in \Pic^d_\alpha(\kc/U)$.} Here we are using that we can assume that all our curves $C$ parametrised by $U$ are contained in the smooth part of $X$.

{Start with any $u$. Then, switching from a twisted line bundle $F$ on a curve $C$ in $U$ to the kernel $F'$ of some surjection $F\twoheadrightarrow k(x)$ results in $\chi(F'\otimes u)=\chi(F\otimes u)-\rk(u)$. Note that in the process $d$ is changed to $d-(1/r)$. Also note that the minimal rank of an $\alpha$-twisted locally free sheaf $E$ on $X$, the index of $\alpha$, has the same prime factors as $|\alpha|$, cf.\ \cite[Thm.\ 6]{AntWil}.

Next we explain how to change $u$. For this consider any invertible sheaf $\kl$ on
the hypersuface $Y\subset X_{\rm sm}$ for which $\Br(X)\to\Br(Y)$ is trivial, see the begining of Section \ref{sec:setup}. Then, as $\alpha|_Y$ is trivial, $\kl$ can be considered as an $\alpha|_Y$-twisted sheaf on $Y$ and its push-forward $i_\ast\kl$ as an $\alpha$-twisted sheaf on $X$. Since $Y\subset X$ is contained in the smooth locus of $X$, there exists a finite locally free resolution 
$G^\bullet\to i_*\kl$ of $i_\ast\kl$
by locally free $\alpha$-twisted sheaves $G^i$, cf. \cite{Cald}. Then
$$\chi(F\otimes u-\sum (-1)^i[G^i])=\chi(F\otimes u)-\chi(F\otimes i_\ast \kl)=
\chi(F\otimes u)-(C.Y)$$
for any locally free sheaf $F$ concentrated on a curve $C$ intersecting $Y$ transversally. Observe that $(C.Y)$ is just the degree of the hypersurface $Y$, i.e.\ $m\coloneqq (h^{n-1}.Y)$.\smallskip

Now, under the assumption that $|\alpha|$ and $m$ are coprime, we can repeat the procedure changing $F$ and $u$ and eventually
produce $F$ and $u$ as needed. Note that the original $d$ might have changed in the process but always such that $\Pic^d_\alpha(\kc/U)$ is still non-empty. }


\subsection{Removing the coprimality assumption}\label{sec:proof}
So far we have proved Theorem \ref{thm:main} for all Brauer classes
$\alpha\in \Br(X)$ such that the order $|\alpha|$ is coprime to the top intersection number $m=(h^{n-1}.Y)$, where
$|h|$ is a very ample linear system on $X$. We now show how to
avoid this numerical assumption by passing to a blow-up of $X$
and thus conclude the proof of Theorem \ref{thm:main}.\smallskip

{Fix a smooth point $x\in Y\subset X$ in the distinguished hypersurface $Y$. Then consider the blow-up $\sigma\colon X'\coloneqq{\rm Bl}_x(X)\to X$ together with the strict transform $Y'\subset X'$  of $Y$ which is the blow-up $Y'={\rm Bl}_x(Y)\to Y$ and which has the property that $\Br(X')\to \Br(Y')$ is trivial. 

If we denote by $E$ the exceptional divisor of $\sigma$, then
$|h'\coloneqq a\cdot \sigma^\ast (h)-E|$ is a very ample linear system on $X'$ for $a\gg0$.\TBC{This is rather clear by first doing this for $\PP^n$. Do we say more? There is a paper
by Ballico and Coppens but I would rather not cite this.}
The degree of $Y'$ with respect to this very ample linear system is
$$m'\coloneqq (h'^{n-1}.Y')=a^{n-1}(h^{n-1}.Y)+(-1)^{n-1}(E|_{Y'})^{n-1}=a^{n-1}\cdot m\pm 1,$$
for $E|_{Y'}$ is the exceptional divisor of $\sigma|_{Y'}\colon Y'\to Y$.
Hence, $m=(h^{n-1}.Y)$ and $m'=(h'^{n-1}.Y')$ are coprime.}\smallskip

According to our previous discussion, there exist two abelian varieties
$A$ and $A'$ over $K(X)=K(X')$, such that every $\alpha\in \Br(X)$ 
with $(|\alpha|,m)=1$ is split by an $A$-torsor $B_\alpha$ and every $\alpha'\in
\Br(X')$ with $(|\alpha'|,m')=1$ is split by an $A'$-torsor $B'_{\alpha'}$.

Since $\Br(X)\cong\Br(X')$, this can be applied as follows.
We write any  given class $\gamma\in \Br(X)\cong\Br(X')$ as a product
$$\gamma=\alpha\cdot\alpha'$$
with $(|\alpha|,m)=1$ and $(|\alpha'|,m')=1$.
Then $\gamma$ is split by the product $B_\alpha\times B_{\alpha'}$ which is
a torsor for the abelian variety $A\times A'$.

This concludes the proof of Theorem \ref{thm:main}.

\subsection{Comments on the assumptions} Our construction is geometric and we need our ground field to
be algebraically closed. For example, even in the untwisted version, a Poincar\'e line bundle may not exist if $k$ is not algebraically closed. More geometrically,
if $X$ is a smooth projective curve over a field which is not algebraically closed,
then $\Br(X)$ might be non-trivial but the construction of system of curves does
not make sense.

As our arguments use the existence of the (twisted) Picard variety  of curves contained in the variety $X$, it seems unlikely that the ideas could be extended to 
also cover ramified classes in $\Br(K(X))$, i.e.\ those that are not contained in
$\Br(X)$.

\TBC{Gebhard's remark that we actually construct a map $\Br\to WC(A)$. It is not totally obvious as the degree $d$ seems to depend on $\alpha$.}

{\begin{remark}\label{rem:gen2} As a continuation of 
the discussion in Section \ref{sec:genK3} and especially Remark \ref{rem:gen1},
we note that the techniques above can be modified to prove similar results
for varieties $X$ with a dominating family of curves $\kc\to X$, potentially of much lower genus than the complete intersection curves used above. In the vain
of Proposition \ref{prop:singell}, one would then prove splitting of Brauer classes
over torsors for abelian varieties $A$ over some finite extension $K'/K(X)$ of the function field, but again with a certain condition on the period of the Brauer class.
\end{remark}}


\section{Period--index problem}\label{sec:proofInd}

This section contains the proof of Theorem \ref{thm:mainind}.
As in the previous section, we can assume that there exists a
smooth hypersurface $Y\subset X$ such that $\Br(X)\to \Br(Y)$
is trivial or, in fact, that $\Br(Y)$ is trivial.

\subsection{Bounding the index of central simple algebras}\label{sec:alg}
The main argument to prove Theorem \ref{thm:mainind} can
be given using twisted sheaves, sheaves on Brauer--Severi varieties, or
sheaves on gerbes. We chose to present it in the language of Azumaya algebras,
as it makes the argument most transparent. We begin by recalling the following basic fact.

Assume $A$ is an Azumaya algebra over a field $K$ and $V$ is a module over $A$. Then $$(\ind(A)\cdot d(A))\mid \dim_K(V).$$ Indeed, writing
$A\cong M_\ell(D)$ and using Morita equivalence, we know that
$V$ is of the form $W^{\oplus \ell}$ for some module $W$ over the division algebra $D$.
Since any $D$-module is free, this proves $V\cong D^{\oplus \ell \cdot m}$ for some $m$ and, therefore, $\ind(A)\cdot d(A)=\dim_K(D)^{1/2}\cdot \dim_K(D)^{1/2}\cdot \ell$
divides $\dim_K(V)=\dim_K(D)\cdot (\ell\cdot m)$.

So, in order to prove Theorem \ref{thm:mainind}, it suffices to find for each
Azumaya algebra $\ka$ on $X$  an $A$-module $V_\ka$  with $\dim_K(V_\ka)=d(A)\cdot\per(A)^e$ for a certain fixed $e$. Here,  $A=\ka_{K}$
denotes the Azumaya algebra over the function field $K=K(X)$ obtained as the generic fibre of $\ka$. \smallskip

The basic idea is to produce such an $A$-module $V$,
cf.\ (\ref{eqn:defV}), as the space of global sections 
\begin{equation}\label{eqn:theta}
H^0(\Pic_\alpha^d(\kc_0/U_0)_K,\kp|_{\Pic_\alpha^d(\kc_0/U_0)\times {\eta_{\kc_0}}}\otimes M)
\end{equation}
of the restriction of the universal sheaf $\kp$ twisted by an appropriate line bundle $M$ and to control its dimension. {So, ultimately $V$ is a space of theta functions
on a twisted Picard variety.} Since $\kp$ is a 
sheaf of modules over the pull-back of $\ka$, its space of global sections is 
indeed an $A$-module.

{\begin{remark} The above can also be phrased in more geometric terms.
Assume $E$ is a locally free sheaf on  $X$ and a module over the Azumaya
algebra $\ka$.  If $\rk(E)=d(\ka)$, then the natural injection $\ka\,\hookrightarrow \kend(E)$ is generically an isomorphism. Moreover, since both sheaves have
trivial determinant, it is an isomorphism in codimension one and hence,
at least when $X$ is smooth, everywhere. In particular, the class of $\ka$ is trivial. 

\end{remark}}

{The basic idea to construct modules over Azumaya algebras (or twisted sheaves) of small rank to bound the index has been exploited before by Lieblich \cite[Prop.\ 4.2.1.3]{LiebComp}. Instead of trying to produce rational points of moduli spaces to produce those, we propose here to use spaces of global sections of
certain twisted sheaves.}

\subsection{Direct image of the twisted Poincar\'e bundle} With the notation of the last section, we consider a birationally dominating family of very ample complete intersection curves 
$$\xymatrix{\kc_0\subset{\rm Bl}_{\{x_i\}}\ar@{->>}[r]&X}$$ parametrised by an open subset $U_0\subset \PP^{n-1}$.\smallskip

As in Section \ref{sec:finems}, we shall first assume that for some $d$
there exists a universal sheaf  $\kp$ on $\Pic^d_\alpha(\kc_0/U_0)\times_{U_0}\kc_0$ over $U_0$. If $\Pic^d_\alpha(\kc_0/U_0)$ is viewed as
a moduli space of sheaves over the Azumaya algebra $\ka$ picked to represent
the class $\alpha\in  \SBro(X)$, then the universal sheaf $\kp$ is a module over
the pull-back of $\ka$ under the projection  $\Pic^d_\alpha(\kc_0/U_0)\times_{U_0}\kc_0\to \kc_0\to X$. As always, the universal sheaf is only unique up to twist coming from $\Pic^d_\alpha(\kc_0/U_0)$ and the first
step consists of twisting $\kp$ appropriately so that it parametrises
(the analogue of) degree zero line bundles on $\Pic_\alpha^d(\kc_0/U_0)$.
\smallskip

To simplify notation, let us restrict the situation to the generic point
$\xi\in U_0$. We denote the fibre of $\kc_0$ by $C$, a curve over $K_0\coloneqq k(\xi)=K(U_0)$, and set $$P\coloneqq \Pic_\alpha^d(\kc_0/U_0)_{\xi}=\Pic_{\alpha|_C}^d(C),$$
which is a torsor  for the abelian variety $\Pic^0(C)$ over $K_0$.\smallskip

{Intersecting $C$ with one of the exceptional divisors $E_i\subset{\rm Bl}_{\{x_i\}}(X)$ produces a closed point $x\in C$, which can also be viewed as the generic point of $E_i$. As $E_i\cap\kc_0$ is a section
of $\kc_0\to U_0$, the residue field of $x$ is $k(x)=K_0$.
Moreover, the pull-back $A_x$ of $\ka$, which is the pull-back of $\ka\otimes k(x_i)$
under the projection $E_i\to \{x_i\}$, is Morita trivial, as $k(x_i)\cong k$ is algebraically closed by assumption. Hence, $A_x\cong M_{d(A)}(K_0)$.\smallskip

The restriction $\kp|_{P\times\{x\}}$ of the universal sheaf can then be seen as a 
a locally free sheaf of modules on $P\cong P\times\{x\}$ over the Azumaya algebra $\ko_P\boxtimes A_x\cong M_{d(A)}(\ko_{P\times\{x\}})$ and is thus of the form $L^{\oplus d(A)}$ for some
invertible sheaf $L$ on $P\times\{x\}$. Thus, tensoring the twisted sheaf  $\kp$
on $P\times C$ by  the invertible sheaf $L^\ast\boxtimes\ko$ results in a universal sheaf
$\kp'$ on  $P\times C$ for which the restriction
to $P\times\{x\}$ is $\ko^{\oplus d(A)}$.}

If $M$ is any ample invertible sheaf on $P$, then for $i>0$ $$H^i(P\times\{x\},(\kp'\otimes(M\boxtimes \ko))|_{P\times\{x\}})=H^i(P\times\{x\},(M\boxtimes \ko)^{\oplus d(A)}|_{P\times\{x\}})=0,$$  {which by semi-continuity implies the same vanishing for the restriction to $P\times\{\eta\}$, where $\eta$ is the generic point of $C$, whose residue field is the function field $K=K(X)$.

The next step is to find an appropriate $M$ for which $\chi(P\times\{\eta\},(\kp'\otimes
(M\boxtimes \ko))|_{P\times\{\eta\}})$ }can be bounded. For this we consider the natural morphism
\begin{equation}\label{eqn:finitemorph}
P=\Pic_{\alpha|_C}^d(C)\to
\Pic^D_{\alpha^{\per(\alpha)}|_C}(C)
\end{equation}
given by tensor product (over $\ka|_C$) $E\mapsto E^{\otimes \per(\alpha)}$,
where $D=\per(\alpha)\cdot d$. As $\alpha^{\per(\alpha)}=1$, the twisted Picard variety on the right hand side
is actually untwisted, i.e.\ isomorphic $\Pic^a(C)$ for a certain $a$. Since {$C(K_0)\ne\emptyset$, 
the latter is isomorphic to $\Pic^0(C)$ and admits an ample invertible sheaf  $\ko(\Theta)$ satisfying $(\Theta)^{g(C)}=g!$, i.e.\ an invertible sheaf inducing a principal polarisation.

Thus, (\ref{eqn:finitemorph}) induces a finite surjective morphism $$\varphi\colon P\twoto \Pic^0(C)$$
of degree $\per(\alpha)^{2g(C)}$ and the pull-back $M\coloneqq\varphi^\ast\ko(\Theta)$
is ample with  top intersection number $(M)^{g(C)}= (\Theta)^{g(C)}\cdot\deg(\varphi)=g!\cdot\per(\alpha)^{2g(C)}$. For this $M$ we then have
by Riemann--Roch and standard vanishing results for abelian varieties
\begin{eqnarray*}&&\dim_{K} H^0(P\times\{\eta\},(\kp'\otimes
(M\boxtimes \ko))|_{P\times\{\eta\}})
=h^0(P\times\{\eta\},(\kp'\otimes
(M\boxtimes \ko))|_{P\times\{\eta\}})\\
&=&\chi(P\times\{\eta\},(\kp'\otimes
(M\boxtimes \ko))|_{P\times\{\eta\}})
=\chi(P\times\{x\},(\kp'\otimes(M\boxtimes \ko))|_{P\times\{x\}})\\
&=&\chi(P\times\{x\},(M\boxtimes \ko)^{\oplus d(A)}|_{P\times\{x\}})
=d(A)\cdot \per(\alpha)^{2g(C)}.
\end{eqnarray*}
Thus, the $A$-module {\begin{equation}\label{eqn:defV}V\coloneqq H^0(P\times\{\eta\},(\kp'\otimes
(M\boxtimes \ko))|_{P\times\{\eta\}})
\end{equation}} is of dimension
$$
\dim_K(V)= d(A)\cdot \per(\alpha)^{2g(C)},
$$ }which by the discussion in Section \ref{sec:alg}
 implies 
\begin{equation}\label{eqn:indper}
\ind(A)\mid\per(\alpha)^{2g(C)}.
\end{equation}

Hence, under the assumption of the existence of a universal family
this proves Theorem \ref{thm:mainind}.


\subsection{End of proof of Theorem \ref{thm:mainind}}\label{sec:End} As in Section \ref{sec:proof}, it remains to show how to argue
in the case that for a given $\alpha$, we cannot find a $d$ for which $\Pic_\alpha^d(\kc_0/U_0)$ is fine. We use the same notation and again decompose a Brauer class $\gamma\in \Br(X)$ as $\gamma=\alpha\cdot \alpha'$.
However, this time we pick the decomposition such that addition that $\per(\alpha)$ and $\per(\alpha')$ are coprime.
Then by the arguments above applied to $X$ and the blow-up $X'$, there exist integers $e$ and $e'$ such that
$\ind(\alpha)\mid\per(\alpha)^e$ and $\ind(\alpha')\mid\per(\alpha')^{e'}$.
This implies $$\ind(\gamma)\mid\ind(\alpha)\cdot\ind(\alpha')\mid
\per(\alpha)^e\cdot \per(\alpha')^{e'}\mid (\per(\alpha)\cdot \per(\alpha'))^{e+e'}=\per(\gamma)^{e+e'}.$$
This concludes the proof of Theorem \ref{thm:mainind}.

\end{document}